\documentclass[preprint]{article} 

\usepackage[english]{babel}
\usepackage{amsfonts, amsmath, amssymb, amsthm, amscd}
\usepackage{amsmath}
\usepackage{xcolor}
\usepackage{ulem}
\usepackage{mathtools}
\usepackage{xcolor}
\usepackage{amsthm}
\allowdisplaybreaks

\usepackage{geometry} 
\usepackage{graphicx} 
\usepackage{hyperref}
\usepackage{cancel}
\hypersetup{pdfstartview=FitH,  linkcolor=blue,urlcolor=urlcolor, citecolor=blue, colorlinks=true}

\newtheorem{theorem}{Theorem}[section]
\newtheorem{lemma}[theorem]{Lemma}
\newtheorem{definition}[theorem]{Definition}
\newtheorem{cor}[theorem]{Corollary}
\newtheorem{rem}[theorem]{Remark}

\newcommand{\Label}[1]{\label{#1}}

\newcommand\RR{\mathbb{R}}
\newcommand\QQ{\mathbb{Q}}
\newcommand\NN{\mathbb{N}}

\newcommand\CC{\mathbb{C}}

\newcommand{\cP}{\mathcal{P}}
\newcommand{\cF}{\mathcal{F}}

\newcommand{\cD}{\mathcal{D}}
\newcommand{\ZZ}{\mathbb{Z}}

\newcommand{\dd}{\partial}
\newcommand{\al}{\alpha}

\newcommand{\ga}{\gamma}
\newcommand{\OO}{\Omega}
\newcommand{\pr}{\prime}
\newcommand{\la}{\lambda}

\newcommand{\Zp}{\mathbb Z_p}
\newcommand{\Qp}{\mathbb Q_p}

\newcommand{\vph}{\varphi}

\newcommand{\Fx}{\cF_{x\to \xi}}

\newcommand{\w}[1]{\widetilde{#1}}
\newcommand{\wh}[1]{\widehat{#1}}

\newcommand{\col}{\colon}
\newcommand{\vep}{\varepsilon}

\newcommand{\DAN}{D^\alpha_N}

\date{}

\numberwithin{equation}{section}

\begin{document}


\title{Existence and uniqueness for $p$-adic counterpart of the porous medium equation}

\author{
	\textbf{Alexandra V. Antoniouk}\\
\footnotesize Institute of Mathematics of the National Academy of Sciences of Ukraine,\\
\footnotesize Tereshchenkivska 3, Kyiv, 01024 Ukraine,\\
\footnotesize American University Kyiv\\
\footnotesize Poshtova Sq 3, Kyiv, 04070 Ukraine
\footnotesize E-mail: antoniouk.a@gmail.com
\and
\textbf{Anatoly N. Kochubei}\\
\footnotesize Institute of Mathematics of the National Academy of Sciences of Ukraine,\\
\footnotesize Tereshchenkivska 3, Kyiv, 01024 Ukraine,\\
\footnotesize E-mail: kochubei@imath.kiev.ua 
\and
\textbf{Oleksii L. Nikitchenko}\\
\footnotesize The Ohio State University,\\
\footnotesize 231 W 18th Ave, Columbus, OH 43210, \\
\footnotesize Kyiv Academic University of the National Academy of Sciences of Ukraine\\
\footnotesize Vernadsky blvd., 36, 03142, Kyiv, Ukraine\\
\footnotesize E-mail: oleksiinikitchenko94@gmail.com
}

%
%
%
%
%
%
%

\maketitle

\bigskip
\begin{abstract}
We develop a theory of generalized solutions of the nonlinear evolution equations for complex-valued functions of a real positive time variable and $p$-adic spatial variable, which can be seen as non-Archimedean counterparts of the fractional porous medium equation.

In this case, we face the problem that a $p$-adic ball is simultaneously open and closed, thus having an empty boundary. To address this issue, we use the algebraic structure of the field of $p$-adic numbers and apply the Pontryagin duality theory to construct the appropriate fractional Sobolev type spaces.

We prove the existence and uniqueness results for the corresponding nonlinear equation and define an associated nonlinear semigroup.

\end{abstract}




\newpage
\section{Introduction}

By Ostrowski’s theorem \cite[Thm. 1, Ch. I]{Koblitz:1984}, the field of rational numbers $\QQ$ permits only two essentially different topological completion: with respect to the metric topology induced by the usual absolute value of $x\in\QQ$ or ultrametric topology induced by non-Archimedean $p$-adic absolute value.

As a result, there exist two parallel mathematical worlds that evolve entirely independently. While classical differential equation theory, which is founded on the analysis of functions on the real (complex) number field, has already yielded a broad array of mathematical methods and concepts, the theory of such equations in $p$-adic function spaces is not as deeply developed.

Simultaneously, as proposed by Stephen Hawking and other authors, it is suggested that the intricate nature of spacetime at smaller scales could exhibit fractal characteristics, including the possibility of being $p$-adic \cite{GibHaw:1993, Not:1993, WheelerFord:1998, Volovich:2010}.  Therefore, the development of mathematical foundations for such objects is considered exceptionally important.

While the linear theory of such equations has already seen some development (see e.g. \cite{VVZ,Ko:2001,ZG} and the references therein), then nonlinear equations in $p$-adic functional spaces require further development of corresponding techniques. It appears that the first papers in this direction have only appeared relatively recently  \cite{Ko:2018, AKK}. At the same time, there is a lot of unsolved questions.

A non-Archimedean counterpart of the fractional porous medium equation, that is the equation
\begin{equation}\Label{1-1}\
D_tu+D^\al \big(\vph(u)\big)=0,\ \ u=u(t,x),\ \ t>0,\ x\in\Qp,
\end{equation}
was first studied in the paper \cite{Ko:2018}. Here $\Qp$ is the field of $p$-adic numbers, $D^\al$, $\al >0$ is Vladimirov's fractional differentiation operator
\begin{equation}\Label{DA}\
\big(D^\al u\big) (x)=\frac{1-p^\al}{1-p^{-\al-1}}\int\limits_{\Qp}\dfrac{u(x-y)-u(x)}{\vert y\vert_p^{\al+1}}\,dy
\end{equation}
acting on complex-valued functions on $\Qp$. Above in \eqref{1-1} $\vph$ is a strictly monotone increasing smooth continuous function, such that $\vph^\pr >0$, $\vph (\pm\infty)=\pm\infty$ and $\vph(0)=0$. We keep in mind the example $\vph(u)=\vert u\vert^{m-1}u$, $m>0$. Equation \eqref{1-1} is a nonlinear evolution equation for complex-valued functions of a real positive time variable and a $p$-adic spatial variable.

The additional complexity for such problem lies in the non-local nature of the corresponding pseudo-differential operator. Even ordinary non-local problems for real-variable functions are the subject of modern challenging research (see e.g. \cite{Vazquez:2007, Bonforte:2015} and references therein).

Developing an $L^1$-theory of Vladimirov’s $p$-adic fractional differentiation operator, in \cite{Ko:2018} it was proved the $m$-accretivity of the corresponding nonlinear operator and obtained the existence and uniqueness of a mild solution. That was motivated both by the existence of a $p$-adic model of a porous medium \cite{KOJ} and by the fact that the equation (\ref{1-1}) was the first natural example of a strongly nonlinear $p$-adic pseudo-differential equation.

\vspace{2mm}

In this paper, we study a stronger notion of solvability for the Cauchy problem on the $p$-adic ball, which serves as a model for a bounded open set in the $p$-adic context:
\begin{equation}\Label{1-2}
\left\{
\begin{array}{lc}
D_tu + \DAN \big(\vph(u)\big)=0,&t\in[0,T];\\
u(0)=u_0,&
\end{array}
\right.
\end{equation}
with the Vladimirov operator $\DAN$ on the $p$-adic ball
$$B_N=\{x\in \Qp\col \Vert x\Vert_p\leq p^N\}$$ defined, as in \cite{VVZ} by restricting the Vladimirov operator $D^\al$ to functions $u_N$ supported in the ball $B_N$ and considering the resulting function $D^\al u_N$ only on the ball $B_N$.

We prove the existence and uniqueness of generalized solutions of the problem (\ref{1-2}), de\-ve\-lop\-ing with the appropriate modifications, the method by Brezis \cite{Brezis:1971}; see also \cite{Barbu:2010}. A necessary prerequisite having independent interest is the investigation of a ``proper'' $p$-adic analog for Sobolev spaces, which is suitable for the study of boundary value problems (BVP) in the non-Archimedean case.
Sobolev spaces play an extraordinary role in the theory of partial differential equations since they give a natural and important tool for the investigation of different properties of BVP. For the $p$-adic analysis, different authors \cite{Taib,RVZG:2010, GorkaKostrRey:2014, GorkaKostr:2015, GorkaKost:2020} proposed suitable candidates for the role of Sobolev spaces based on the analogy or in the framework of pseudodifferential equations approach.

At the same time, there is a fundamental difference between the study of boundary value problems in the Archimedean and non-Archimedean situations: the finite ball in $p$-adic field has no boundary in the usual sense. That's why the study of $p$-adic BVP requires elaboration of a special and natural approach to the description of the spaces, which should play the role of the analog to the Sobolev spaces in usual PDE analysis.

For this purpose,  in Section \ref{sec3} we prove the equivalence of different analogs of $p$-adic Sobolev spaces and their connection with the domain of operator restricted to a $p$-adic ball. The investigation of this case is based on a purely non-Archimedean effect -- a $p$-adic ball centered at the origin is an additive group, so that we can use the harmonic analysis provided by the Pontryagin duality theory.

Building upon this developed technique, we can establish our main result Theorem \ref{main} on the existence and uniqueness of the weak solution to the nonlinear boundary value problem in $p$-adic ball and state the properties of the corresponding nonlinear semigroup.

\section{Preliminaries}

{\it 2.1. $p$-Adic numbers} \cite{Ko:2001,VVZ}. Let $p$ be a prime
number. The field of $p$-adic numbers is the completion $\mathbb Q_p$ of the field $\mathbb Q$
of rational numbers, with respect to the absolute value $|x|_p$
defined by setting $|0|_p=0$,
$$
|x|_p=p^{-\nu }\ \mbox{if }x=p^\nu \frac{m}n,
$$
where $\nu ,m,n\in \mathbb Z$, and $m,n$ are prime to $p$. $\Qp$ is a locally compact topological field.

Note that by Ostrowski's theorem there are no absolute values on $\mathbb Q$, which are not equivalent to the ``Euclidean'' one,
or one of $|\cdot |_p$.

The absolute value $|x|_p$, $x\in \mathbb Q_p$, has the following properties:
\begin{gather*}
|x|_p=0\ \mbox{if and only if }x=0;\\
|xy|_p=|x|_p\cdot |y|_p;\\
|x+y|_p\le \max (|x|_p,|y|_p).
\end{gather*}

The latter property called the ultra-metric inequality (or the non-Archi\-me\-dean property) implies the total disconnectedness of $\Qp$ in the topology
determined by the metric $|x-y|_p$, as well as many unusual geometric properties. Note also the following consequence of the ultra-metric inequality:
\begin{equation*}
|x+y|_p=\max (|x|_p,|y|_p)\quad \mbox{if }|x|_p\ne |y|_p.
\end{equation*}

The absolute value $|x|_p$ takes the discrete set of non-zero
values $p^N$, $N\in \mathbb Z$. If $|x|_p=p^N$, then $x$ admits a
(unique) canonical representation
\begin{equation}
\Label{2.1}\
x=p^{-N}\left( x_0+x_1p+x_2p^2+\cdots \right) ,
\end{equation}
where $x_0,x_1,x_2,\ldots \in \{ 0,1,\ldots ,p-1\}$, $x_0\ne 0$.
The series converges in the topology of $\mathbb Q_p$. For
example,
$$
-1=(p-1)+(p-1)p+(p-1)p^2+\cdots ,\quad |-1|_p=1.
$$

The {\it fractional part} of element $x\in\Qp$ in canonical representation \eqref{2.1} is given by:
\[
\{x\}_p=
\left\{\begin{array}{ll} 0,&\text{if}\quad N\leq 0 \ \ \ \text{or}\ \ \ x=0;\\
p^{-N}\big(x_0+x_1p+\ldots+x_{N-1}p^{N-1}\big),&\text{if} \quad N>0.
\end{array}
\right.
\]
The function $\chi(x)=\exp(2\pi i\{x\}_p)$ is an additive character of the field $\Qp$, i.e. the character of its additive group. Let us remark that $\chi(x) = 1$ if and only if $\vert x\vert_p\leq 1$.

We denote by $dx$ the Haar measure on the additive group $\Qp$ normalized on $\ZZ_p$ with the requirement: $\int_{\ZZ_p} dx = 1$, where $\Zp = \{x\in\Qp\col \vert x\vert_p\leq 1\}$ is the unit ball in $\Qp$.

\medskip
{\it 2.2. Fourier transformation and distributions on $\Qp$}.

Let us denote by $\cD (\Qp)$ the vector space of {\it test functions}, locally constant functions with compact supports.
 Recall that a function $\psi: \Qp\to \CC$ is {\it locally constant} if there exist such an integer $\ell \geq 0$ that for any $x\in\Qp$
\[\psi (x+y)=\psi (x),\quad \text{if}\quad \Vert y\Vert_p\leq p^{-\ell},\quad \text{($\ell$ is independent on $x$).}\]

The smallest number $\ell$ with this property is called {\it the exponent of constancy of the function $\psi$.} Typical examples
of locally constant functions are additive characters, and also cutoff functions like $\OO(\Vert x\Vert_p)$, where
\[\OO(t)=\left\{\begin{array}{ll}
1,& \ \text{if}\ \ 0 \leq t \leq 1;\\
0,& \ \text{if}\ \ t >1.
\end{array}\right.
\]
It is worth remarking that $\OO$ is continuous, which is an expression of the non-Archimedean properties of $\Qp$. Note also that $\cD(\Qp)$ is dense in $L_q(\Qp)$ for each $q\in [1,\infty)$.

Let us also introduce the subspace $D_N^\ell\subset \cD(\Qp)$ consisting of functions with supports in a ball $B_N$, $N\in\ZZ$
and with the exponents of local constancy less than $\ell\in\ZZ$. Then the topology in $\cD(\Qp)$ is defined as the double inductive limit topology, so that
\[\cD(\Qp)=\lim\limits_{\longrightarrow\atop{N\to\infty}}
\lim\limits_{\longrightarrow\atop{\ell\to\infty}}D_N^\ell.\]

If $V\subset \Qp$ is an open set, the space $\cD(V)$ of test functions on $V$ is defined as a subspace of $\cD(\Qp)$ consisting of functions with supports in $V$. For a ball $V=B_N$, we can identify $\cD (B_N)$ with the set of all locally constant functions on $B_N$.

The {\it Fourier transform} of a test function $\psi\in\cD (\Qp)$  is defined by the formula
\[
\big(\Fx\,\psi\big)(\xi)=\int\limits_{\Qp}\chi(\xi\, x)\psi(x)\,dx,\quad \xi \in \Qp.
\]

Remark that the additive group of $\Qp$ is self-dual, so that the Fourier transform of a complex-valued function $\psi\in \Qp$ is again a function on $\Qp$ and if $\Fx\psi\in L_1(\Qp)$ then we have the inversion formula
\[\psi(x)=\int\limits_{\Qp}\chi(-x\, \xi)\Fx\psi(\xi)\,d\xi.\]
Let us also remark that, in contrast to the Archimedean situation,  the Fourier transform
$\psi \to \Fx \psi$
is a linear and continuous automorphism of the space $\cD (\Qp)$ (cf. \cite[Lemma 4.8.2]{AKK:book}, see also \cite[Ch. II,§2.4.]{Gel}, \cite[III,(3.2)]{Taib}, \cite[VII.2.]{VVZ}, i.e.
$
\psi(x)=\cF^{\,-1}_{\xi\to x}\Big(\cF_{x\to\xi} \psi \Big).
$

The space $\cD^\pr(\Qp)$ of Bruhat-Schwartz distributions on $\Qp$ is defined as a strong conjugate space to $\cD(\Qp)$. By duality, the Fourier transform is extended to a linear (and therefore continuous) automorphism of $\cD^\pr(\Qp)$. For a detailed theory of convolutions and direct product of distributions on $\Qp$ closely connected with the theory of their Fourier transforms see \cite{AKK:book,Ko:2001,VVZ}

For the needs of this article, we also require some facts from the harmonic analysis on the $p$-adic ball $B_N$. At first, remark that $B_N$ is a compact subgroup of $\Qp$ and its annihilator $\{\xi\in\Qp\col \chi(\xi\,x)=1\  \text{for all}\  x\in B_N\}$ coincides with the ball $B_{-N}$. By the duality theorem (see, e.g. \cite[Theorem 27]{Morris:1977}, the dual group $\widehat{B}_N$ to $B_N$ is isomorphic to the discrete group $\Qp/B_{-N}$ consisting of cosets
\[p^m (r_0+r_1p+\cdots + r_{N-m-1}p^{N-m-1})+B_{-N}, \ \  r_j\in \{0,1, \ldots, p-1\}, \ m\in \ZZ, \  m<N.\]
This isomorphism means that any nontrivial continuous character of $B_N$, which has the form $\chi (\xi\cdot x)$, $x\in B_N$, where $\vert \xi \vert_p > p^{-N}$ and $\xi \in \Qp$, is considered as a representative of the class $\xi + B_{-N}$. Moreover the value $\vert \xi\vert_p$ is the same for any representative of the class.

Let us recall that the normalized Haar measure on $B_N$ is $d\mu=p^{-N}\, dx$. The normalization of the Haar measure on $\Qp/B_{-N}$ can be made in such a way that the equality
\[\int\limits_{\Qp} f(x)\, dx = \int\limits_{\Qp / B_{-N}}\Big (p^N \int\limits_{B_{-N}} f(x+h)\, dh\Big)\, d\hat{\mu} (x+B_{-N})
\]
is valid for any $f\in \cD (\Qp)$. (See, for example, \cite[Chapter VII, Proposition 10]{Bourbaki:2004} or \cite[Th. (28.54), p. 51(91)]{Hewitt-Ross:1979:II}). Above $\hat{\mu}(x+B_{-N})$ denotes the normalized Haar measure on $\Qp / B_{-N}$. With this normalization it is also true the Plancherel identity for the corresponding Fourier transform on $B_N$ given by the formula:
\begin{equation}\Label{FN}\
\wh{f}\equiv(\cF_N f)(\xi) = p^{-N}\int\limits_{B_N} \chi(x\, \xi)f(x)\, dx, \quad \xi\in (\Qp / B_{-N})\cup \{0\}.
\end{equation}
It follows from \eqref{FN} that $\cF_N f$ can be understood as a function on $\Qp / B_{-N}$.
Since $\cF\col \cD(\Qp)\to \cD (\Qp)$, the Fourier transform $\cF$ maps $\cD(B_N)$ onto the set of functions on the discrete set $\wh{B_N}$ with only a finite number of nonzero values. Thus the set $\cD(B_N)$ with natural locally convex topology may be considered as the set of test functions on $\wh{B_N}=\Qp / B_{-N}$. The conjugate space $\cD^\pr(\wh{B_N})$ consists of all functions on $\wh{B_N}$, see e.g. \cite{Helem:2006}. Therefore the Fourier transform is extended, via duality, to the mapping from $\cD^\pr(B_N)$ to $\cD^\pr(\wh{B_N})$ and the theory of distributions on locally compact group, in particular $B_N$, developed by Bruhat \cite{Bruhat:1961} as well applicable.

The corresponding Plancherel identity has the form (see \cite[(31.46)]{Hewitt-Ross:1979:II}):
\begin{equation}\Label{31-46}\
\dfrac{1}{p^N}\int\limits_{B_N}\vert h\vert^2\,dx=\Vert \cF_N h\Vert^2_{L_2(\widehat B_N)},
\end{equation}
for $h\in L_2(B_N)$.

\medskip
{\it 2.3 Spectrum of the Vladimirov operator on the $p$-adic ball}. In \cite{Kochubei:2018:Ball} it was proven that Vladimirov operator $D_N^\al$ has the representation:
\begin{equation}\Label{1-3}
\left( \DAN u\right) (x)=\lambda_0 u (x)+\frac{1-p^{\alpha}}{1-p^{-\alpha -1}}\int\limits_{B_N}|y|_p^{-\alpha -1}[u (x-y)-u (x)]\,dy
\end{equation}
on functions $u$ from the space $\mathcal D(B_N)$ of locally constant functions with compact support in $B_N$.

Operator $D_N^\al$ is a positive definite self-adjoint operator on $L_2(B_N)$ and
\begin{equation}\Label{lla}\
\lambda_0 =\frac{p-1}{p^{\alpha +1} -1}p^{\alpha (1-N)}
\end{equation}
is its smallest eigenvalue. It is also well known that operator $\DAN$ on $B_N$
has a complete orthonormal system of eigenfunctions, consisting of so called Vladimirov functions (see e.g. \cite[Ch.3, \S 3.3.2]{Ko:2001}, \cite[Section 10.4]{VVZ}):
 \begin{equation*}
	\varPsi_0 (x) =
	\begin{cases}
		p^{-N/2}, &\text{if $x \in B_N$}\\
		0, &\text{if $x \notin B_N$},
	\end{cases}
\end{equation*}
corresponding to the eigenvalue $\la_0$
of multiplicity 1, and
$${\varPsi_{1-N,j,0}^1(x) = p^{\frac{-N}{2}}\Omega(p^{-N} \Vert x \Vert_p) \chi(jp^{N-1}x)} \, ,$$
where ${j \in \{1,\ldots,p-1\}} \,, j \notin \Zp^0,$ with eigenvalue
$\lambda_1=p^{\alpha(1-N)}\ \ (\text{multiplicity} \ p-1),$

$${\varPsi_{\mu-N,j,\varepsilon}^{l}(x) = p^{\frac{\mu-N-l}{2}}\sqrt{\frac{p}{p-1}}\delta(\Vert x \Vert_p - p^{l+N-\mu}) \chi(\varepsilon p^{l-2(\mu-N)}x^2+jp^{l+N-\mu-1}x) \, ,}$$
where $2 \leq l \leq \mu, j \in \{1,\ldots,p-1\}, j \notin \Zp^0, \varepsilon = \varepsilon_0 + \varepsilon_1 \beta + ... + \varepsilon_{l-2} \beta^{l-2}$ and
$$\varPsi_{\mu-N,j,0}^{1}(x)=p^{\frac{\mu-N-1}{2}}\Omega(p^{\mu-N-1}\Vert x \Vert_p)\chi(jp^{N-\mu}x), $$
 where $j \in \{1,\ldots,p-1\},$ $ j \notin \Zp^0,$ with
the eigenvalue $\lambda_{\mu}=p^{\alpha(\mu-N)}$ of the total multiplicity
 $p^{\mu-1}(p-1), \mu = 2,3,...$.
Here {$\Zp^0 = \{ x \in B_N \colon  \Vert x \Vert_p < 1\}$, and}
\begin{equation*}
 	\delta(t) =
 	\begin{cases}
 		1, &\text{if $t=0$}\\
 		0, &\text{if $t\neq 0$}.
 	\end{cases}
 \end{equation*}
For other systems of the orthogonal eigenvectors see e.g. \cite{Koz:2007, BGPW:2014, BZ:2019}.
\section{Sobolev spaces over $p$-adic ball} \Label{sec3}

Consider the additive group of $B_N$. The dual group $\widehat{B}_N$ is isomorphic to the discrete group $\Qp /B_{-N}$ consisting of the cosets
\begin{equation}\Label{K1}\
\xi = p^m (r_0+r_1p+\ldots+r_{N-m-1}p^{N-m-1}) + B_{-N},
\end{equation}
where $r_j \in \{0,1,\ldots,p-1\}$, $m\in\ZZ$, $m<N.$ For $\xi\in \Qp /B_{-N}$ we set $\Vert \xi\Vert = p^{-m}$.

\begin{definition}\Label{defSob}\ The Sobolev space $H^\al(B_N)$ consists of such functions $f\in L_2(B_N)$ that
\[\Vert f\Vert^2_{H^\al(B_N)}=\int\limits_{\widehat B_N}\vert \widehat f(\xi)\vert^2\big(1+\vert \xi\vert^2_p\big)^\al\,d\xi=\sum\limits_{\xi\in\widehat{B}_N}\vert \widehat{f}(\xi)\vert^2 (1+\Vert \xi\Vert^2)^\al < \infty\]
where $\widehat{f}=\cF_N f$ is the Fourier transform in the ball $B_N$ \eqref{FN}.
\end{definition}

\begin{definition}\Label{defAGS}\ Let $u\in L_2(B_N)$ and $s\in (0,1)$. We say that function $u$ belongs to the Aronszain-Gagliardo-Slobodecki space $H_{AGS}^s(B_N)$ if the following norm is finite:
\[\Vert u\Vert_{H_{AGS}^s}=\Vert u\Vert_{L_2(B_N)}+[u]_s,\]
where
\begin{equation}\Label{Us}\
[u]_s^2=\int\limits_{B_N}\int\limits_{B_N}\frac{\vert u(x)-u(y)\vert^2}{\vert x-y\vert_p^{2s+1}}\,dx\,dy.
\end{equation}
\end{definition}

For the case $\RR^n$ the corresponding space was introduced by Aronszain, Gagliardo and Slobodecki independently in \cite{Aron:1955, Gagl:1958, Slob:1958}. The properties of these spaces for the locally compact abelian groups were recently investigated in \cite{GorkaKost:2020}. Below we prove some results about these spaces, which we apply further.
\begin{theorem} \Label{prop-3-3}\ If $0<s<1$ then the spaces $H^s_{AGS}(B_N)$ and $H^s(B_N)$ are isomorphic, i.e. there are some constants $C_1$ and $C_2$ such that
\begin{equation}\Label{CC1}\
C_2\Vert u\Vert_{H^s(B_N)}\leq \Vert u\Vert_{H^s_{AGS}(B_N)} \leq C_1\Vert u\Vert_{H^s(B_N)}.
\end{equation}
\end{theorem}
\begin{proof} In fact,
\[[u]^2_s=\sum\limits_{\xi\in \widehat{B}_N}\vert \wh{u}(\xi)\vert^2 A_s(\xi),\]
where
\begin{equation}\Label{2K}\
A_s(\xi)=\int\limits_{B_N}\dfrac{\vert \chi(z\,\xi)-1\vert^2}{\vert z\vert_p^{2s+1}}\, dz.
\end{equation}
Indeed, by substituting $z= x-y$ into \eqref{Us} and using Fubini theorem and then Plancherel equality we have:
\begin{align*}
[u]^2_s&=\int\limits_{B_N}\int\limits_{B_N}\frac{\vert u(z+y)-u(y)\vert^2}{\vert z\vert_p^{2s+1}}\,dy\,dz=\int\limits_{B_N}\Big\Vert \dfrac{u(z+\cdot)-u(\cdot)}{\vert z\vert^{s+1/2}}\Big\Vert^2_{L_{2}(B_N)} dz=\\
&=\int\limits_{B_N}\Big\Vert \cF_N\Big(\dfrac{u(z+\cdot)-u(\cdot)}{\vert z\vert^{s+1/2}}\Big)\Big\Vert^2_{L_{2}(\wh{B}_N)} dz=\int\limits_{B_N}\Big\Vert \dfrac{\chi(z\,\xi) \wh{u}(\xi)-\wh{u}(\xi)}{\vert z\vert^{s+1/2}}\Big\Vert^2_{L_{2}(\wh{B}_N)} dz=\\
&=\int\limits_{B_N}\sum\limits_{\xi\in\wh{B}_N} \dfrac{\vert\chi(z\, \xi) -1\vert^2}{\vert z\vert^{2s+1}}\vert \wh{u}(\xi)\vert^2 dz=\sum\limits_{\xi\in \widehat{B}_N}\vert \wh{u}(\xi)\vert^2 A_s(\xi).
\end{align*}

Similar to the inequality (13) from \cite{GorkaKost:2020}, we may prove that there exists such a constant $C_1$ that
\begin{equation}\Label{3-4}\
A_s(\xi)\leq C_1\Vert \xi\Vert^{2s}, \quad \forall \xi \in \wh{B}_N.
\end{equation}
This implies that
\[[u]_s\leq C_1\sum\limits_{\xi\in\wh{B}_N}\vert \wh{u}(\xi)\vert^2\Vert \xi\Vert^{2s}\leq C_1 \Vert u\Vert^2_{H^s(B_N)}.\]
To prove left part of the inequality \eqref{CC1} let us make the change of variables in \eqref{2K}, where $\xi \in \wh{B}_N$ is identified with its principal part in \eqref{K1}, $\vert \xi\vert_p\geq p^{-N+1}$. We write $z=\eta\,\xi^{-1}$, so that
\begin{align*}
A_s(\xi)&=\int\limits_{\vert \eta\vert\leq p^N\Vert\xi\Vert}\dfrac{\vert\chi(\eta)-1\vert^2}{\vert \eta\vert_p^{2s+1}\Vert\xi\Vert^{-2s-1}}\Vert \xi\Vert^{-1}\,d\eta=\\
&=\Vert \xi\Vert^{2s}\int\limits_{\vert \eta \vert_p\leq p^N\Vert \xi\Vert}\dfrac{\vert \chi(\eta)-1\vert^2}{\vert \eta\vert_p^{2s+1}}\, d\eta.
\end{align*}
Since $\Vert \xi\Vert \geq p^{-N+1}$, we have
\[\big\{\vert \eta\vert_p \leq p^N\Vert \xi\Vert \big\}\supset\big\{\vert \eta \vert_p\leq p\big\},\]
therefore
\[A_s(\xi)\geq C_2 \Vert \xi\Vert^{2s},\]
where
\[C_2 = \int\limits_{\vert \eta\vert_p\leq p}\dfrac{\vert\chi(\eta)-1\vert^2}{\vert \eta\vert_p^{2s+1}}\,d\eta.\]
\end{proof}

\section{The existence of a weak solution}

As we mentioned in Section 2.3 operator $D_N^\al$ \eqref{1-3} is a self-adjoint operator on $L_2(B_N)$ with discrete spectrum. Let us denote by $\{\mathfrak{a}_k\}, k = 1, 2, \ldots $, $\mathfrak{a}_k >0$ its eigenvalues written in increasing order and repeating according to their multiplicity, and by $\{\psi_k\}_{k\geq 1}$ corresponding set of eigenfunctions, normalized in $L_2(B_N)$. They form an orthogonal basis in $L_2(B_N)$ (see e.g. \cite{Koz:2007, BGPW:2014, BZ:2019}).
\begin{definition}\Label{def1-2}\
Let us define a Hilbert space $H_1$ as the space of all functions $u\in L_2(B_N)$ such that
\begin{equation}\Label{H1}\
u=\sum\limits_{k=1}^\infty c_k \psi_k \in L_2(B_N), \quad \Vert u \Vert^2_{H_1} = \sum\limits_{k=1}^\infty \mathfrak{a}_k \vert c_k\vert^2 <\infty.\end{equation}
Thus the dual space $H_{-1}$ to $H_1$ is defined as the completion of the finite sums of the form $$f=\sum\limits_{k=1}^Nc'_k\psi_k$$ with respect to the dual norm
	$$||f||^2_{H_{-1}}=\sum\limits_{k=1}^{\infty} \mathfrak{a}_k^{-1}|c'_k|^2.$$
\end{definition}

The operator $D_N^{\alpha}$ in $L_2(B_N)$, defined on the dense domain $\mathcal D(B_N)$ of locally constant functions with compact support in the $p$-adic ball $B_N$,  is semi-bounded from below

$$( D_N^{\alpha} u,u)_{L_2(B_N)} \geq \la_0\Vert u \Vert_{L_2(B_N)}$$
for any $u \in \mathcal D(B_N)$. Therefore the Hilbert space $H_1$ is isomorphic to the closure of  $\mathcal D(B_N)$ with respect to the scalar product
\begin{equation}\Label{H1-1}\
(u,v)_{1} = (u,v)_{L_2(B_N)}+(D_N^{\alpha}u, v)_{L_2(B_N)}.
\end{equation}

Let operator $\cP_{N,\al}$ acting on $\cD(B_N)$ given by
\begin{equation}\Label{PNal}\
\cP_{N,\al}u(x)=D_N^\al u(x)-\la_0u(x)=\dfrac{1-p^\al}{1-p^{-\al-1}}\int\limits_{B_N}\dfrac{u(x-y)-u(x)}{\vert y\vert^{\al+1}_p}\,dy,
\end{equation}
with $\la_0$ from \eqref{lla}.
\begin{theorem}\Label{Th-hitchhiker}\
\begin{equation}\Label{PN}\
\cF_N \big(\cP_{N,\al}u\big)(\xi)=\frac{1}{p^N}\vert\xi\vert^\al_p\cF_Nu(\xi),\quad  \xi\in\widehat B_N, \al\in(0,1).
\end{equation}
\end{theorem}
\begin{proof}
Let us find $\cF_N\big(\cP_{N,\al}u\big)$.
\begin{align}\nonumber
\cF_N\big(\cP_{N,\al}u\big)&=\dfrac{1}{p^N}\int\limits_{B_N}\chi(x\xi)\cP_{N,\al}u(x)\,dx=\\
\nonumber
&=\dfrac{a_p}{p^N}\int\limits_{B_N}\chi(x\xi)\int\limits_{B_N}\dfrac{u(x-y)-u(x)}{\vert y\vert_p^{\al+1} }\,dy\,dx=\\
\nonumber
&=\dfrac{a_p}{p^N}\int\limits_{B_N}\dfrac{1}{\vert y\vert_p^{\al+1} }
\int\limits_{B_N}\chi(x\xi)[u(x-y)-u(x)]\,dx\,dy=\\
\nonumber
&=a_p\int\limits_{B_N}\dfrac{\chi(y\,\xi)-1}{\vert y\vert_p^{\al+1} }
\,dy \,\cF_N u(\xi),
\end{align}
where $\xi\in\Qp/B_{-N}$ and $a_p=\dfrac{1-p^\al}{1-p^{-\al-1}}$.
Let us prove that

\begin{equation}\Label{norm-alpha}\
 \frac{1-p^\al}{1-p^{-\al-1}}\int\limits_{B_N}\dfrac{\chi(y\,\xi)-1}{\vert y\vert_p^{\al+1} }
\,dy =\vert \xi\vert^\al_p, \quad \xi\in \widehat{B}_N, \al\in(0,1).
\end{equation}

To prove \eqref{norm-alpha} let us consider the Riesz kernel
\[f_\ga^N(x)=\dfrac{\vert x\vert_p^{\ga-1}}{\Gamma_p(\ga)}, \quad x\in B_N, \ga \in (0,1),\]
where $\Gamma_p(\ga)=\dfrac{1-p^{\ga-1}}{1-p^{-\ga}}$ denotes $p$-adic Gamma function (see e.g. \cite[Ch. II, \S 2.6 (14)]{Gel}). Its Fourier transform equals to:
\begin{align}\nonumber\
\cF_Nf_\ga^N(\xi)&=\dfrac{1}{p^N \Gamma_p(\ga)}\int\limits_{B_N}\chi(x\xi)\vert x\vert_p^{\ga-1}\, dx=\\
\nonumber
&=\dfrac{1}{p^N \Gamma_p(\ga)}\int\limits_{\vert y\vert_p\leq p^N\vert \xi\vert_p}\chi(y)\vert y\vert_p^{\ga-1}\vert \xi\vert_p^{-\ga}\, dy=\\
\nonumber\
&=\dfrac{\vert \xi\vert_p^{-\ga}}{p^N \Gamma_p(\ga)}\int\limits_{\vert y\vert_p\leq p^N\vert \xi\vert_p}\chi(y)\vert y\vert_p^{\ga-1}\, dy,\quad \ga\in(0,1),
\end{align}
$\xi\in \Qp / B_{-N}\simeq\widehat{B}_N$.
Above we used the change of variables $x=y\xi^{-1}$, therefore $dx=\vert\xi\vert_p^{-1}dy.$ Remark that due  to (12.40) in \cite{Vladimirov:tables} for any $r\geq 1$:
\[\int_{B_r}\vert x\vert_p^{\ga-1}\chi(x)\,dx=\Gamma_p(\ga).\]
Since $\vert \xi\vert_p=p^m$ with $m<N$, see \eqref{K1}, we have that $p^N\vert \xi\vert_p>1$.

Therefore considering $f_\ga^N(x)$ as a distribution from $\cD^\pr(B_N)$ we may identify the Fourier transform $\cF_N$ of $f_\ga^N$ with
\[\w f_\ga^N(\xi)=p^{-N}\vert \xi\vert_p^{-\ga}, \quad\xi \in \widehat{B}_N.\]

Now for any $\vph\in \cD(\widehat B_N)$ we denote $\psi=\cF^{-1}_N\vph$ and applying the regularization theory for homogeneous distributions \cite{Gel-Shilov} (see also \cite{VVZ, Ko:2001}) we find that for $\ga \in (0,1)$:
\begin{align*}
\langle\vert \xi\vert^{-\ga}_p,\vph(\xi)\rangle_{L_2(\wh{B}_N)}&= p^N\,\langle \w f_{\ga}^N,\vph(\xi)\rangle_{L_2(\wh{B}_N)}= p^N\,\langle \cF_N f_{\ga}^N,\cF_N\psi\rangle_{L_2(\wh{B}_N)}=\\
&=\frac{ p^{N}}{\Gamma_p(\ga)}\, \langle \vert x\vert_p^{\ga-1},\psi (x)-\psi(0)\rangle_{L_2(B_N)}=\\
&=\frac{ p^{N}}{\Gamma_p(\ga)}\, \int\limits_{B_N}\vert x\vert_p^{\ga-1}\int\limits_{\widehat B_N}\big(\chi(x\xi)-1\big)\vph(\xi)\,d\xi\,dx=\\
&=\int\limits_{\widehat B_N}\vph(\xi)\int\limits_{B_N}\frac{ p^{N}}{\Gamma_p(\ga)}\vert x\vert_p^{\ga-1}\big(\chi(x\xi)-1\big)\, dx\,d\xi.
\end{align*}
Therefore in terms of distributions we have for $\ga \in (0,1)$, $\xi \in \widehat B_N$:
\begin{equation}\Label{4-7}\
\dfrac{1}{\Gamma_p(\ga)}\int\limits_{B_N}\vert x\vert_p^{\ga-1}\big(\chi(-x\xi)-1\big)\, dx =\vert \xi\vert_p^{-\ga}.
\end{equation}
The distribution $f_\ga^N(x)=\dfrac{\vert x\vert_p^{\ga-1}}{\Gamma_p(\ga)}$ is holomorphic on $\ga$ everywhere except poles $1+\al_k$, $\al_k=\dfrac{2k\pi}{\ln p}i$ (e.g. \cite[Ch.I, \S 8]{VVZ}). Thus both sides of the equality \eqref{4-7} are well-defined for $\ga\in (-1,0)$ and also for $\ga = -\al$. Finally applying \eqref{norm-alpha} we have the statement.
\end{proof}
\begin{theorem}\Label{isom}\ Let $\al\in(0,1)$. Then
\begin{equation}\Label{eq-prop-3-4}\
\Vert u\Vert^2_{H^\al(B_N)}\asymp \Vert u\Vert^2_{H^\al_{AGS}}\asymp \Vert u\Vert ^2_{H_1}.
\end{equation}
\end{theorem}
\begin{proof} The first part of the isomorphism relation \eqref{eq-prop-3-4} follows from Theorem \ref{prop-3-3}. To prove second part first remark that
\[\Vert u \Vert_{H_1}\asymp\Vert u\Vert^2_{L_2(B_N)}+\Vert [\cP_{N,\al}]^{1/2}u\Vert^2_{L_2(B_N)}.\]
From Definition \ref{defSob} it follows that $u\in H^\al(B_N)$ is equivalent to $\vert \xi\vert_p^\al\, \widehat u(\xi)\in L_2(\widehat B_N)$. Finally from \eqref{PN} we have the result.
\end{proof}

The dual space $H_{-1}$ to $H_1$ is obtained as the closure of $L_2(B_N)$ w.r.t. the scalar product
\begin{equation}\Label{sc} \
(f,g)_{-1}=(f,g)_{L_2(B_N)}+([D_N^{\alpha}]^{-1}f, g)_{L_2(B_N)}.
\end{equation}

The dual space $H_{-1}$ is also called the distribution space. For the construction of the dual space in terms of unbounded self-adjoint operator with discrete spectrum acting in some Hilbert space we refer to \cite[Ch. XIV]{BER}.


If $u,v \in H_1$ and $f=(\DAN u)(x)$ then since $\DAN$ acts as an isomorphism between $H_1$ and $H_{-1}$ we have
$$(f,v)_{H_{-1}\times H_1}=([\DAN]^{-1}f,v)_{H_1\times H_1}=(u,v)_{H_1\times H_1}=\sum\limits_{k=1}^{\infty} \mathfrak{a}_k u_kv_k.$$
If $f \in L_2(B_N)$, then $ (f,v)_{H_{-1}\times H_1}= \int_{B_N}fv\,dy$ for $v\in H_1$.

\begin{definition}\Label{def-4-4}
A function $u \in C([0,T], H_{-1})$ is called an $H_{-1}$ - solution to equation \eqref{1-2} if $\varphi(u) \in L_1([0,T], H_1)$ and

$$\Label{1-8}\ \int\limits_0^T\int\limits_{B_N}u\,D_{t}\psi\,dy\,dt = \int\limits_{0}^T\int\limits_{B_N}\varphi(u)\,\DAN\psi \,dy\,dt$$ for any $\psi \in \cD(B_N)$, where $\cD(B_N)$ is the space of locally constant functions with compact support in $B_N$.
\end{definition}
Our main result is the following theorem about the existence of a weak solution to nonlinear problem \eqref{1-2} in sense of Definition \ref{def-4-4} for any $T>0$.
\begin{theorem}[{\bf Main result}]\Label{main}
For any $u_0 \in H_{-1}$ there exists a unique solution $u \in C([0,T], H_{-1})$ of problem \eqref{1-2} for every $T>0$. Moreover we have
$$t \vph(u) \in L_{\infty}([0,T],H_{-1})$$
$$t\,\partial_tu \in L_{\infty}([0,T],H_{-1})$$
We also have that $u\vph(u) \in L_1([0,T]\times B_N )$ and
the solution map $S_t: u_0 \rightarrow u(t)$ defines a semigroup of non-strict contractions in $H_{-1}$
\begin{equation}\Label{eq4-5}\
\Vert u(t)-v(t)\Vert_{H^{-1}}\leq\Vert u(0) - v(0)\Vert_{H^{-1}},
\end{equation}
{which turns out to be also compact in $H_{-1}$}.

\end{theorem}

Let us introduce some notations. Let $j:~\RR \rightarrow \RR^+$ be a convex, lower semi-continuous function such that its subdifferential $\partial j=\varphi$, $j(0) = 0$ and $j(r)/|r| \rightarrow \infty $ as $|r| \rightarrow \infty$.
For $u\in H_{-1}$ we define
$$\Psi(u)=\int_{B_N}j(u)\,dx,$$
if $u \in L_1(B_N) $ and $j(u) \in L_1(B_N)$, and $\Psi(u):= +\infty$ otherwise.

Let us recall that the {\it subdifferential of function} $j:~\RR \rightarrow \RR^+$ is given by
\begin{equation}\Label{gradj}\
\dd  j(u) =\big \{\rho \in \RR\col\, j(v)-j(u)\geq \rho\cdot(v-u), \ \forall v \in Dom(j) \},
\end{equation}
where $Dom(j) = \{ x \in \RR \col\, j(x)< +\infty\}$.

The proof of the main result Theorem \ref{main} follows the principal approach of Brezis \cite{Brezis:1971} with necessary modifications dictated by specific features of the $p$-adic situation, where classical properties of Sobolev spaces require a separate investigation. The proof is given in a set of lemmas.

\begin{lemma} The function $\Psi(u)$ is convex lower semi-continuous on $H_{-1}$.
\end{lemma}
\begin{proof}[Proof]
Let $u_n$ be a sequence of function such that $u_n \in H_{-1}\cap L_1(B_N)$, $u_n \to u $ in $ H_{-1}$ as $n \to \infty$ and $\int_{B_N}j(u_n(x))\,dx \leq C$.

Remark that for each $n \in \NN$ $\int_{B_N} u_n(x)\,dx$ is  uniformly absolutely continuous,
i.e. $\forall \vep>0\   \exists\,\delta>0 $ such that  $\forall E\col \ \mu(E)<\delta$ implies $\int_{E}u_n(x)\,dx < \varepsilon$.

Indeed, let $A>\frac{2C}{\vep}$ and $R$ be such that $j(r)/|r| \geq A$ for $|r|>R$. Let $\delta< \frac{\varepsilon}{2R}$, then we have
\begin{align}\nonumber
\int_{E}|u_n|\,dx &\leq \int\limits_{\{ x\ \in E \col |u_n(x)|\geq R \}}|u_n(x)|\,dx + \int\limits_{\{ x\ \in E \col |u_n(x)|< R \}}|u_n(x)|\,dx \leq\\
\nonumber
&\leq \int_{B_N}\frac{j(u_n(x))}{A}\,dx + R\delta \leq \frac{C}{A}+ R\delta < \varepsilon.
\end{align}
{By Dunford-Pettis theorem \cite[Theorem 9, Ch. IV, \S 8]{Dun}, which is valid for any measurable space, there is a subsequence such that $u_{n_k} \to \tilde{u}$ weakly in $L_1(B_N)$}.
Since we know that  $u_n \to u$ in $H_{-1}$ as $n\to \infty$, we conclude that $u_n \to u$ weakly in $L_1(B_N)$.
Finally the function $\Psi(u)$ is {convex} and by Fatou's lemma, $\int\limits_{B_N}j(u)\,dx \leq \liminf\limits_{n \to \infty}\int\limits_{B_N}j(u_n)\,dx$, it is lower semi-continuous function on $L_1(B_N)$ {and thus is weakly lower semi-continuous on $H_{-1}$}.
\end{proof}

\begin{cor}  The sub-differential $\partial\Psi$ is a maximal monotone operator in $H_{-1}$.
\end{cor}
\begin{proof}The function $\Psi$ is convex and lower semi-continious in $H_{-1}$, so that its sub-differential $\partial\Psi$ is a maximal monotone operator in $H_{-1}$ \cite{Min}.
\end{proof}
Let us define operator $A$ on $ H_{-1}$ by the following:
\begin{equation}\Label{Anon}\
	Au = \{\DAN w\col \  w \in H_1,\, w(x) \in \varphi(u(x))\ \text{for a.e.}\  x\in B_N \}
\end{equation}
with $u\in Dom (A)$ iff there is some $w\in H$ such that $w(x)\in \vph (u(x))$ for a.e. $x\in B_N$.

Our main goal is to prove that $A$ is maximal monotone, then this gives the main statement of Theorem \ref{main} on the existence of the solution. Inequality \eqref{eq4-5} is the consequence of maximal monotonicity by standard arguments similar to \cite{Kato:1970}(see also \cite{Komura:1967, CrandallPasy:1969, Brezis:1971}). The maximal monotonicity of $A$ will follow from the fact that $A \subset \partial\varPsi$. To prove the letter we need the following result.

\begin{theorem}\Label{L1-6}
\[f \in \partial\Psi(u) \Leftrightarrow [\DAN]^{-1}(f)(x) \in \varphi(u(x)) \text{\ for a.e \ } x\in B_N. \]

\end{theorem}

\begin{lemma}\Label{L1-10}
Let $F\in H_{-1} \cap L_1(B_N)$ and let $w \in H_1$. Let $g \in L_1(B_N)$ and $h$ be  measurable function such that
\begin{equation}\ \Label{28}\
g \leq h \leq F \cdot w, \ \text{for a.e.}\   x\in B_N.
\end{equation}
 Then $h \in L_1(B_N)$ and
 \begin{equation}\Label{43}\
 \int_{B_N}h\,dx \leq \int_{B_N}F(x)\cdot w(x)\,dx = (F,w)_{H_{-1}\times H_1  }
 \end{equation}
 holds. Here  $(F,w)_{H_{-1}\times H_1}$ means the duality pairing between $H_{-1}$ and $H_1$.
\end{lemma}
\begin{proof}[Proof]
Let \begin{equation}\Label{1-14-0}\
w_n =
 \begin{cases}
   n, &\text{if $w \geq n$}\\
   w, &\text{if $|w| \leq n$}\\
   -n, &\text{if $w \leq -n$}
 \end{cases}
\end{equation}
and $h_n := h\frac{w_n}{w}, g_n :=g\frac{w_n}{w}$.

Let us remark that $w_n\in H_1$ for $w\in H_1$. To check this first of all note that $w_n\in L_2(B_N)$ for $w\in H_1$ due to the inequality: $\vert w_n\vert \leq \vert w\vert$ a.e. on $B_N$. Moreover, by Definition~\ref{def1-2} of $H_1$ $w_n\in H_1$ iff $\vert w_n\vert \in H_1$. For the orthogonal basis $\{\psi_k\}_{k\geq 1}$ in $L_2(B_N)$ the Fourier coefficients $c_k = (\vert w\vert ,\psi_k)_{L_2(B_N)}$ and $c_k^n=(\vert w_n\vert ,\psi_k)_{L_2(B_N)}$ in the decompositions
$$\vert w\vert = \sum\limits_{k=0}^{\infty} c_k \psi_k, \quad \vert w_n\vert = \sum\limits_{k=0}^{\infty} c^n_{k} \psi_{k}$$
are connected by the following inequality $\vert c^n_{k} \vert \leq \vert c_k \vert.$ Thus  $w_n \in H_1$, if $w\in H_1$. Moreover the sequence $\{\vert w_n\vert \}_{n\geq 1}$ and therefore $\{w_n\}_{n\geq 1}$ are bounded in $H_1$.

Multiplying inequality \eqref{28} by $\frac{w_n}{w}$, we get
$$g_n \leq h_n\leq F\cdot w_n \  \text{for a.e.}\  x\in  B_N.$$
Therefore
$$0\leq h_n-g_n\leq F\cdot w_n - g_n\ \text{ for a.e.}\  x \in B_N.$$
Since $w_n \to w$ a.e. on $B_N$, we have $h_n - g_n \rightarrow h-g$ a.e. on $B_N$ as $n\to \infty$. We also have that
\begin{equation}\Label{wn}\
\int_{B_N}(h_n-g_n)\,dx\leq \int_{B_N}F\cdot w_n\, dx -\int_{B_N} g_n \,dx= (F,w_n)_{H^{-1}\times H} - \int_{B_N} g_n \,dx.
\end{equation}

Further, we will apply Theorem 11 in \cite{GorkaKostrRey:2014}, which in our notations states the following. Since $\dfrac{1}{1+\vert \xi\vert_p^2}\in L_q(\widehat B_N)$ for any $q>1$, we have for $\al\in (0,1)$ (thus $\al < q$) the compact embedding of the space $H^\al (B_N)$ into $L_r(B_N)$ for all $2<r<q^*$, where $q^*=\frac{2q}{q-\al} > 2$.

From the definition of $w_n$ \eqref{1-14-0} we have that
\begin{equation}\Label{eq14-1}\
\vert w_n(x)-w_n(y)\leq \vert w(x)-w(y)\vert, \quad \text{\ for all}\ x,y \in B_N.
\end{equation}
Therefore, due to the Theorem \ref{isom} the sequence $\{w_n\}_{n\geq 1}$ is uniformly bounded in $H^\al_{AGS}$, and we may assume its weak convergence. Moreover, it follows from definition that $\frac{w_n}{w}\to 1$ as $n\to\infty$ and $\vert\frac{w_n}{w}\vert\leq 1$. Therefore $w_n\to w$, a.e. on $B_N$ as $n\to\infty$. Using the isomorphism properties (Theorem \ref{isom}) and the above-mentioned compact embedding, we have that $w_n\to w$, $n\to\infty$ in $L_r(B_N)$ strongly, and therefore in $L_2(B_N)$ $(r>2)$.

Using \eqref{eq14-1} together with the dominated convergence theorem in $H^\al_{AGS}$ we can pass to the limit. Together with the strong $L_2$ convergence this gives the strong convergence in $H^\al_{AGS}$ and therefore in $H_1$. We have also that $g_n\to g$, $n\to \infty$ in $L_1(B_N)$ and from \eqref{wn} by Fatou's Lemma we conclude that $h-g \in L_1(B_N)$ and thus $h\in L_1(B_N)$ together with
$$\int_{B_N}(h-g)\,dx\leq (F,w)_{H_{-1} \times H_1} - \int_{B_N}g\,dx<\infty,$$
which implies the statement.
\end{proof}

\begin{lemma}\Label{L1-11}
$A \subset \partial\Psi.$
\end{lemma}
\begin{proof}[Proof]
To prove this statement we need to show that for any $f$ such that $f = Au$ for some $u\in Dom(A)$, it is also true $f\in \dd \Psi$. Recalling the definition of $\dd \Psi$:
\[\,\dd\Psi (u) =\big \{g\in H_{-1}\col\, \Psi (v)-\Psi (u)\geq (g,v-u)_{H^{-1}}, \ \forall v \in \cD(\Psi)\big\},\]
where $Dom(\Psi)=\{u\in H_{-1}\col \Psi(u)< +\infty\}$, we conclude
that it suffices to prove that for each $f= Au$ (i.e. such that there
exists $u \in H_{-1}\cap L_1(B_N)$ with $f =\DAN w$,
where $w\in H_1$ and $w(x) \in \varphi(u(x))$ for a.e $x \in B_N$) we have
$$ \int_{B_N}j(v)\,dx -\int_{B_N}j(u)\,dx \geq (f,v-u)_{H_{-1}}=(w,v-u)_{H_1\times H_{-1} }$$
for any $v\in Dom(\Psi)$. Above we have used that $\DAN$ acts as an isomorphism between $H_1$ and $H_{-1}$.

To prove Lemma \ref{L1-11} for $v\in Dom(\Psi)$ let us take $v \in  H^{-1}
\cap L_1(B_N)$ and $j(v) \in L_1(B_N)$. Since $w \in \varphi(u)$
and $\varphi = \partial j$, we have  $w \in \partial j(u)$. Therefore according to the definition of subdifferential (\ref{gradj}) we have $j(v)-j(u)
\geq w\cdot(v-u)$ a.e on $B_N$.

Since $j(r)/|r| \rightarrow \infty  $
as $ |r| \rightarrow \infty$, we can find constants $C_1\ C_2$, such
that $j(r)\geq C_1|r|+C_2$. Let $F=u-v$, $h=j(u)-j(v)$, $g=j(v)-C_1|
u|-C_2$. Then $ F \cdot w \geq  h \geq g $ a.e on $B_N$, and we can
apply Lemma \ref{L1-10}. We get $h \in L_1(B_N)$.  Also,
$(F,w)_{H_{-1} \times H_1} \geq \int\limits_{B_N}h\,dx $. From that,
we get inequality

$$ \int\limits_{B_N}j(u)\,dx -\int\limits_{B_N}j(v)\,dx \geq (v-u,w)_{H_{-1}\times H_1}=(w,v-u)_{H_1 \times H_{-1}},$$
it means that $\Psi(v)-\Psi(u)\geq(w,v-u)_{H_1 \times H_{-1}}= (f,v-u)_{H_{-1}}$,
so $f \in \partial\Psi(u)$.
\end{proof}
This finished the proof of Theorem \ref{L1-6}.
\medskip

\begin{lemma}\Label{L1-13}
$A$ is a maximal monotone operator.
\end{lemma}
\begin{proof}[Proof]
	Due to \eqref{Anon},
we need to proof that for a given $f \in H_{-1}$ there exists $u \in H_{-1}\cap L_1(B_N)$ and $w\in H_1$ such that
\begin{equation}\Label{eq1}\
u+\DAN w =f,
\end{equation}
and $w(x)=\vph(u(x))$ a.e. on $B_N$.
As $\vph$ is a strictly monotone increasing smooth con\-ti\-nu\-ous function we introduce $\eta := \varphi^{-1}$ and rewrite equation \eqref{eq1} as
\begin{equation} \Label{eq2}\
 \eta(w)+ \DAN w=f.
 \end{equation}
Let $\eta_{\mu}=\dfrac{1}{\mu}(1-J_\mu)$ be the Yosida approximation of the nonlinear map $\eta$, where its resolvent $J_\mu$ is given by $J_\mu = (1+\mu \eta)^{-1}$. Remark that for any $\mu \in \RR^+$ there is a solution $w_{\mu}\in H_1$ of equation
\begin{equation} \Label{eq3}\
\eta_{\mu}(w_\mu)+ \DAN w_\mu=f.
\end{equation}

Indeed, let us denote $\w\DAN = \DAN - \frac{\lambda_0}{2}$, where $\lambda_0$ is the minimal non-zero eigenvalue of operator $D_N^\alpha$ (see e.g. \cite[Chapter II  10 4]{VVZ}).
Since $\w{\DAN}$ is a maximal monotone operator,
 due to the Browder Theorem \cite[Th. 2]{Bro} we have that operator $\eta_{\mu}+\w{\DAN}$ is maximal monotone. Then from Minty Theorem \cite{Min} it follows that  the map $\eta_{\mu}+\w{\DAN}+\frac{\lambda  I}{2}$ is surjective. Therefore equation
\[
\eta_{\mu}(w_\mu)+ \w{\DAN} w_\mu+\frac{\lambda}{2}w_\mu=f
\]
has a unique solution for any right-hand side $f \in H_{-1}$ and this implies solvability of equation \eqref{eq3}.

The equation \eqref{eq3} may be rewritten in form
\begin{equation} \Label{eq4}\
	\DAN(w_\mu)=\mathfrak{f}_\mu,
\end{equation}
where $\mathfrak{f}_\mu = f- \eta_{\mu}(w_\mu)$.
Let us show that $w_\mu$ is uniformly bounded in $H_1$ with respect to $\mu \to 0$.

Firstly, operator $(\DAN)^{-1} : H_{-1} \to H_1$ is continuous. This follows from closed graph theorem. For this we need to check that operator $(\DAN)^{-1} : H_{-1} \to H_1$ has closed graph, i.e from convergence $\varphi_k \to \varphi$ in $H_{-1}$ and $(\DAN)^{-1}\varphi_k \to f $ in $H_1$ it follows that $(\DAN)^{-1}\varphi = f $. Indeed, from these conditions, due to continuity of embedding $H_{1} \subset H_{-1}$ (see Appendix) we have that $(\DAN)^{-1}\varphi_k \to f $ in $H_{-1} $. Since operator $(\DAN)^{-1} : H_{-1} \to H_{-1}$ is continuous (see Appendix), we have that $(\DAN)^{-1}\varphi_k \to (\DAN)^{-1}\varphi $ in $H_{-1}$, therefore $(\DAN)^{-1}\varphi =f$ as required to prove.

Secondly, $\mathfrak{f}_\mu$ is uniformly bounded in $H_{-1}$ as $\mu \to 0$.
Indeed, due to \eqref{eq4}
\begin{equation} \Label{eq5}\
	\Vert \mathfrak{f}_\mu \Vert_{H_{-1}} \leq \Vert f \Vert_{H_{-1}} + \Vert \eta_\mu(w_\mu) \Vert_{H_{-1}} \leq C + \Vert \eta_\mu(w_\mu) \Vert_{H_{-1}},
\end{equation}
and uniform boundness of $\eta_\mu(w_\mu)$ in $H_{-1}$ as $\mu \to 0$ follows from Theorem 2.1 \cite{BCP}.

 The continuity of operator $(\DAN)^{-1} : H_{-1} \to H_1$ and the uniform boundedness of $\mathfrak{f}_\mu$ in $H_{-1}$ as $\mu \to 0$ due to
$$
 w_\mu = (\DAN)^{-1}(\mathfrak{f}_\mu)
$$(see \eqref{eq4}), gives the uniform boundedness of $w_\mu$ in $H_1$ with respect to $\mu \to 0$.

Thus we can find a subsequence $\mu_n \to 0$ such that $w_{\mu_n} \to w$ in $H_1$ and in $L_2(B_N)$,  i.e.
\begin{align} \Label{eq8}\
 w_{\mu_n} \to w &\text{ a.e. on } B_N,\\  \Label{eq8-1}\
(I+\mu_n\eta)^{-1}w_{\mu_n} \to w &\text{ a.e. on } B_N.
\end{align}
Indeed, by \cite[Lemma 1.3.d]{BCP}
$$J_{\mu_n} (w):=(I+\mu_n\eta)^{-1}w \to w$$ as $\mu_n \to 0$ in $H_1$. Due to the inequality $\Vert J_\mu(x)-J_\mu(y)\Vert_{H_1}\leq \Vert x - y\Vert_{H_1}$  (see e.g. \cite[Prop. 2.3]{Barbu:2010}) we have that $J_{\mu_n}(w_{\mu_n})\to w$ in $H_1$ and therefore \eqref{eq8}.


Multiplying equation \eqref{eq3} by $w_{\mu_n}$, we see that for all $\mu_n > 0$
\begin{align}\Label{eq11}\,
\int\limits_{B_N}\eta_{\mu_n}(w_{\mu_n}) \cdot w_{\mu_n} \, dx
&\leq \int\limits_{B_N}\eta_{\mu_n}(w_{\mu_n}) \cdot w_{\mu_n} \, dx +
\int\limits_{B_N} D_N^\al w_{\mu_n}\cdot w_{\mu_n}\, dx \leq\\
& \leq\Vert f\Vert_{H_{-1}}\cdot\Vert w_{\mu_n}\Vert_{H_1}
\leq C,
\end{align}
 due to uniform boundedness of $w_{\mu_n}$ in $H_1$ with respect to $\mu_n \to 0$.

%
%
%
%

Let us define $g_{\mu_n}:= \eta_{\mu_n}(w_{\mu_n})\subset \eta (w_{\mu_n})$ (see e.g. \cite[Lemma 2.1]{BrPa:1970}). To finish the proof of Lemma \ref{L1-13} we need the following lemma.

\begin{lemma}\Label{Th18}\,
	Let $\eta$ be a maximal monotone graph in $\RR \times \RR$ such that $Dom\,(\eta)=\RR$ and $0 \in \eta(0)$, such that $\vph = \eta^{-1}$ may be represented as $\vph = \dd j$, $j(0)=0$, $Dom\, (j)=\RR.$
	
	Let $g_{n}$ and $w_{n}$ be measurable functions on $B_N$ such that $w_{n} \to w$ a.e. on $B_N$ as $n\to \infty$, $g_{n}\in \eta(w_{n})$ a.e. on $B_N$  and $g_{n}\cdot w_{n} \in L_1(B_N)$ with
	\begin{equation}\Label{eqTh18}\,
	\int\limits_{B_N}g_{n}\cdot w_{n}\, dx \leq C.
	\end{equation}
	Then, there is a subsequence $n_k \to \infty$ such that $g_{n_k} \to u$ weakly in $L_1(B_N)$.
\end{lemma}

\begin{proof}


Since $\varphi = \partial j$, $j(0)=0$ we may write
\begin{equation}\Label{1-20}\
\dd j(u)=\vph(u)=\{a\in \RR \col j(v)-j(u)\geq a\cdot (v-u), \forall v\in Dom\,(j)\}.
\end{equation}
By the definition of the inverse map
\[T^{-1}(x^*)=\{x\in X\col x^*\in T(x)\}\]
for nonlinear map $T\col X\to X^*$ acting from some Banach space $X$ into its dual $X^*$, from inclusion $g_n\in \eta(w_n)$ (see e.g. \cite[Lemma 2.1]{BrPa:1970})  we have that $w_n \in \eta^{-1}(g_n)=\vph(g_n)$. Taking in \eqref{1-20} $v=0$, $u=g_n(x)$, then $a= w_n\in \vph(g_n)$ therefore we have
\[-j(g_n(x))\geq w_n\cdot\big(-g_n(x)\big).\]

Thus from \eqref{eqTh18} we obtain that for all $n\geq 1$
$$\int_{B_N} j(g_{n}(x))\, dx \leq \int_{B_N}g_{n}\cdot w_{n}\, dx \leq C.$$

Applying the Vall\'ee Poussin reformulation of Dunford-Pettis Theorem (\cite[Th.1.3(d), Ch. VIII]{EkelTemam}) to the sequence $\{g_{\mu_n}\}$ we have its weak compactness in $L_1(B_N)$. Therefore, there is a subsequence $n_k \to \infty$ such that
\begin{equation} \Label{eq6}\
g_{\mu_{n_k}} \to \, u \ \text{weakly in} \ L_1(B_N)
\end{equation}
to some $u\in L_1(B_N)$, which finishes the proof of Lemma \ref{Th18}. \end{proof}

To finish the proof of Lemma \ref{L1-13} and  thus the statement of the existence of the solution in the main Theorem, it remains to prove that

\begin{equation}
\Label {eq6-1}\
u(x) \in \eta(w(x)) \  \text{a.e.  on} \ B_N,
\end{equation}
which implies that $w(x) \in \varphi(u(x))$ a.e. on $B_N$.

It suffices to prove that for every $M>0$, $u(x)\in \eta (w(x))$ a.e. for $x\in B_N^M \col = \{x\in B_N\col \vert w(x)\vert \leq M\}.$ From \eqref{eq8} and Egorov theorem,  it follows that for every $\varepsilon >0$ there exists a measurable subset $E_\varepsilon \subset B^M_N$ such that $mes\, (B^M_N \setminus E_\varepsilon) \leq \varepsilon$  and
\begin{equation} \Label{eq7}\
w_{\mu_n} \to w \text{ uniformly in } E_\varepsilon \text{ as } \mu_n \to 0,
\end{equation}
and $w\in L_\infty(B_N)$.

%
Note that operator
$$\tilde{\eta} = \{ [u,w] \in L_1(E_\varepsilon)\times L_{\infty}(E_\varepsilon) \colon u(x) \in \eta(w(x)) \text{ a.e. on } E_\varepsilon \}$$
is maximal monotone in $L_1(E_\varepsilon)\times L_{\infty}(E_\varepsilon)$.
Indeed, let $\tilde{w} \in L_{\infty}(E_\varepsilon)$ and $\tilde{u} \in L_{1}(E_\varepsilon)$ be such that
\begin{equation}\Label{eq9}\
\tilde{u}(x) \in \eta(\tilde{w}(x))\  \text{a.e. on}\ E_\varepsilon.
\end{equation}

Without loss of generality we may assume that $Dom(\eta)$ is bounded. If $Dom(\eta)$ is not bounded, consider $\tilde{\eta}=\eta+\partial I_B$, where $I_B$ is the indicator function of a ball centered at $0$ of large radius:
\[
I_B(x)=\left\{
\begin{array}{lc}
0&x\in B;\\
+\infty& x\notin B.
\end{array}
\right.
\]

By the monotonicity of $\eta$, since $g_{n_k}\in \eta(w_{n_k})$, we have
 $$(\tilde{u}-g_{\mu_{n_k}},\tilde{w}-w_{\mu_{n_k}})\geq 0$$ a.e. on $E_\varepsilon$, thus

  $$\int_{E_\varepsilon}(\tilde{u}-g_{\mu_{n_k}},\tilde{w}-w_{\mu_{n_k}})\, dx \geq 0.$$
Consequently $\int\limits_{E_\varepsilon}(\tilde{u}-u,\tilde{w}-w)\, dx \geq 0$.
Let $\tilde{w} = (I+\eta)^{-1}(w+u)$. Remark that $\tilde{w} \in L_{\infty}(E_\varepsilon)$ since $Dom(\eta)$ is bounded. We have $\tilde{w}+\eta\tilde{w} \in w+u$ a.e. on $E_\varepsilon$.
Choosing

\begin{equation}\Label{eq10}\
\tilde{u}=w+u-\tilde{w}
\end{equation}
we get $\int\limits_{E_\varepsilon}|w-\tilde{w}|^2\,dx \leq 0$ so that $\tilde{w}=w$ a.e.
Therefore from \eqref{eq10} we have $\tilde{u} = u$.
Substituting both of this equalities  in \eqref{eq9} we have $u(x) \in \eta(w(x))$ a.e. on $E_\varepsilon$.
Therefore $w(x) \in \varphi(u(x))$ a.e. on $E_\varepsilon$. Since $\varepsilon$ is arbitrary, we conclude that $w(x) \in \varphi(u(x))$ a.e. on $B_N$.


Therefore, proof of Lemma \ref{L1-13} and thus Theorem \ref{L1-6} is finished. \end{proof}

From maximal monotonicity of operator $\partial \Psi (u)$ in $H_{-1}$ and it's characterization in terms of $\DAN$, we obtain that statement of Theorem \ref{main} holds.

\bigskip
\begin{rem}\rm
The above results remain valid, if we consider, instead of $\Qp$, an arbitrary non-Archimedean local field. In turn, this makes it easy to develop multi-dimensional generalizations of our results.
It is known \cite{Ko:2021,Ko:2023} that the transition to some multi-dimensional operators is equivalent to considering the whole setting over a wider field, an unramified
extension of the initial one.
\end{rem}

 \section {Acknowledgements}
The first author acknowledges the funding support in the framework of the project ``Spectral Optimization: From Mathematics to Physics and Advanced Technology'' (SOMPATY) received from the European Union’s Horizon 2020 research and innovation programme under the Marie Skłodowska-Curie grant agreement No 873071.

\end{document}